\documentclass[a4paper,11pt,openany,oneside,article]{memoir}

\usepackage{etoolbox}

\usepackage[utf8]{inputenc}
\usepackage[T1]{fontenc}
\usepackage[charter,cal=cmcal]{mathdesign}
\usepackage[scaled]{beramono,berasans}
\usepackage{textcomp}
\usepackage{microtype}
\usepackage{xspace}
\midsloppy

\usepackage{amsmath,bm}
\numberwithin{equation}{section}
\usepackage{amsthm}
\usepackage{thmtools}
  \declaretheorem[name=Theorem,within=section]{theorem}
  \declaretheorem[name=Lemma,sibling=theorem]{lemma}
  
  \declaretheorem[name=Definition,sibling=theorem,style=definition]{definition}

\usepackage{mathtools}
\usepackage{colonequals}
\usepackage{algorithm}
\usepackage[noend]{algpseudocode}

\usepackage{xcolor}
  \definecolor{ugentblue}{RGB}{36,71,127}
  \definecolor{ugentyellow}{RGB}{250,178,10}
  \definecolor{ugent-we}{RGB}{131,194,236}
\usepackage{tikz}
\usetikzlibrary{calc,shapes.geometric,shadings}

\usepackage{hyperref}
  \hypersetup{hypertexnames=false,colorlinks,
              linkcolor=ugentblue,citecolor=ugentblue,urlcolor=ugentblue,
              bookmarksdepth=1, bookmarksopen=true,
              pdfstartview={XYZ null null 1}}
  
\usepackage[all]{hypcap}
\usepackage[backend=biber,style=alphabetic,noerroretextools]{biblatex}
\usepackage[noabbrev,nameinlink,capitalise]{cleveref}
\let\etoolboxforlistloop\forlistloop
\usepackage{autonum}
\let\forlistloop\etoolboxforlistloop
\usepackage{csquotes}
\usepackage{multirow}
\usepackage[type={CC},modifier={by},version={4.0}]{doclicense}

\newsubfloat{figure}
\usepackage{enumitem}
\captiondelim{. }
\captionnamefont{\bfseries\small}
\captiontitlefont{\small}
\changecaptionwidth
\captionwidth{\linewidth-4em}

\setsecnumdepth{subsection}
\nonzeroparskip
\addtotheoremprefoothook[proof]{\setlength{\parskip}{0pt}}

\newcommand\foreign[1]{#1}
\makeatletter
\newcommand\xperiod{\@ifnextchar.{}{.\@}}
\makeatother
\newcommand\eg{\foreign{e.g.}\xspace}

\newcommand\ie{\foreign{i.e.}\xspace}

\newcommand\resp{\foreign{resp}\xperiod\xspace}

\newcommand*\dash{\nobreakdash-\hspace{0pt}}

\DeclarePairedDelimiter\abs{\lvert}{\rvert}

\DeclarePairedDelimiterX\Set[2]{\lbrace}{\rbrace}{\,#1\,\delimsize\vert\,#2\,}

\firmlists
\usepackage{booktabs}

\newtoggle{bw}

\tikzset{
  edge/.style={thick},
  vertex/.style={circle,fill,scale=0.5}
}

\iftoggle{bw}{
  \tikzset{
    edge0/.style={very thick,loosely dotted},
    edge1/.style={thick,dashed},
    vertex0/.style={vertex,diamond},
    vertex1/.style={vertex,regular polygon,regular polygon sides=4}
  }
}{
  \definecolor{darkgreen}{RGB}{0,128,0}
  \tikzset{
    edge0/.style={thick,red},
    edge1/.style={thick,darkgreen},
    vertex0/.style={vertex,red},
    vertex1/.style={vertex,darkgreen}
  }
}

\title{Generation of Local Symmetry-Preserving Operations on Polyhedra}
\date{}
\author{
Pieter Goetschalckx \\
\footnotesize Ghent University \\
\footnotesize Krijgslaan 281-S9 \\
\footnotesize 9000 Ghent, Belgium \\
\footnotesize \url{pieter.goetschalckx@ugent.be}
\and
Kris Coolsaet \\
\footnotesize Ghent University \\
\footnotesize Krijgslaan 281-S9 \\
\footnotesize 9000 Ghent, Belgium \\
\footnotesize \url{kris.coolsaet@ugent.be}
\and
Nico Van Cleemput \\
\footnotesize Ghent University \\
\footnotesize Krijgslaan 281-S9 \\
\footnotesize 9000 Ghent, Belgium \\
\footnotesize \url{nico.vancleemput@gmail.com}
}
\hypersetup{pdftitle={Generation of Local Symmetry-Preserving Operations},pdfauthor={Pieter Goetschalckx, Kris Coolsaet, Nico Van Cleemput}}

\begin{document}

\maketitle
\begin{tikzpicture}[remember picture,overlay]
  \node at (current page.south) {\begin{minipage}{1.2\textwidth}\centering\doclicenseLongText\end{minipage}};
\end{tikzpicture}
\vspace{-2.5\baselineskip}

\begin{abstract}
  We introduce a new practical and more general definition of local symmetry-preserving operations on polyhedra. These can be applied to arbitrary embedded graphs and result in embedded graphs with the same or higher symmetry. With some additional properties we can restrict the connectivity, \eg when we only want to consider polyhedra. Using some base structures and a list of 10 extensions, we can generate all possible local symmetry-preserving operations isomorph-free.
\end{abstract}

\section{Introduction}

Symmetry-preserving operations on polyhedra have a long history -- from Plato and Archimedes to Kepler \cite{kepler}, Goldberg \cite{goldberg}, Caspar and Klug \cite{casparklug}, Coxeter \cite{coxeter}, Conway \cite{conway}, and many others. Notwithstanding their utility, until recently we had no unified way of defining or describing these operations without resorting to ad-hoc descriptions and drawings. In \cite{Brinkmann2017} the concept of local symmetry-preserving operations on polyhedra (\emph{lsp operations} for short) was introduced. These are operations that are locally defined -- on the \emph{chamber} level, as explained in the next section -- and therefore preserve the symmetries of the polyhedron to which they are applied. This established a general framework in which the class of all lsp operations can be studied, without having to consider individual operations separately. It was shown that many of the most frequently used operations on polyhedra (\eg dual, ambo, truncate, \dots) fit into this framework.

But of course we sometimes do want to examine the operations individually, \eg to check conjectures on as many examples as possible before we try to prove them, or to find operations with certain properties. We can do this for a few operations by hand, but a computer can do this a lot faster, and in a systematic way such that no operations are missed.

In this paper we shall slightly extend the definition of lsp operation so it can be applied to any graph embedded on a compact closed surface\footnote{All graphs in this paper are embedded graphs, and a subgraph has the induced embedding.}, and at the same time provide a reformulation of these operations as decorations, which will turn out to be easier to use in practice.

\section{Decorations and lsp operations}

Every embedded graph $G$ has an associated chamber system $C_G$ \cite{Dress1987}. This chamber system is obtained by constructing a barycentric subdivision of $G$ by adding one vertex in the center of each edge and face of $G$, and edges from each center of a face to its vertices and centers of edges. These vertices can be chosen invariant under the symmetries of $G$. In $C_G$, each vertex has a type that is 0, 1, or 2, indicating the dimension of its corresponding structure in $G$. Each edge has the type of the opposite vertex in the adjacent triangles. In \cref{fig:barycentric}, the chamber system of the plane graph of a cube is given. The original graph consists of the edges of type 2 in the chamber system.

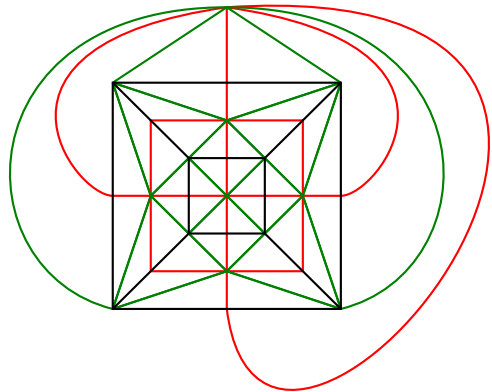
\begin{figure}[htp]
  \centering
  \begin{tikzpicture}[scale=0.5]
    \useasboundingbox (-5,-2.5) rectangle (11,8.5);

    \coordinate (v1) at (0,6);
    \coordinate (v2) at (6,6);
    \coordinate (v3) at (0,0);
    \coordinate (v4) at (6,0);
    \coordinate (v5) at (2,4);
    \coordinate (v6) at (4,4);
    \coordinate (v7) at (2,2);
    \coordinate (v8) at (4,2);

    \coordinate (e1) at (3,6);
    \coordinate (e2) at (6,3);
    \coordinate (e3) at (3,0);
    \coordinate (e4) at (0,3);
    \coordinate (e5) at (3,4);
    \coordinate (e6) at (4,3);
    \coordinate (e7) at (3,2);
    \coordinate (e8) at (2,3);
    \coordinate (e9) at (1,5);
    \coordinate (e10) at (5,5);
    \coordinate (e11) at (5,1);
    \coordinate (e12) at (1,1);

    \coordinate (f0) at (3,8);
    \coordinate (f1) at (3,5);
    \coordinate (f2) at (5,3);
    \coordinate (f3) at (3,1);
    \coordinate (f4) at (1,3);
    \coordinate (f5) at (3,3);

    \draw[edge0] (f0) -- (e1) -- (f1) -- (e5) -- (f5);
    \draw[edge0] (f0) .. controls (10,7) and (7,3) .. (e2) -- (f2) -- (e6) -- (f5);
    \draw[edge0] (f0) .. controls (-4,7) and (-1,3) .. (e4) -- (f4) -- (e8) -- (f5);
    \draw[edge0] (f0) .. controls (18,9) and (4,-8) .. (e3) -- (f3) -- (e7) -- (f5);
    \draw[edge0] (f1) -- (e10) -- (f2) -- (e11) -- (f3) -- (e12) -- (f4) -- (e9) -- (f1);

    \draw[edge1] (f0) -- (v1) (f0) -- (v2)
                 (f0) .. controls (-4,8) and (-4,1) .. (v3)
                 (f0) .. controls (10,8) and (10,1) .. (v4);
    \draw[edge1] (f1) -- (v1) (f1) -- (v2) (f1) -- (v5) (f1) -- (v6);
    \draw[edge1] (f2) -- (v2) (f2) -- (v4) (f2) -- (v6) (f2) -- (v8);
    \draw[edge1] (f3) -- (v3) (f3) -- (v4) (f3) -- (v7) (f3) -- (v8);
    \draw[edge1] (f4) -- (v1) (f4) -- (v3) (f4) -- (v5) (f4) -- (v7);
    \draw[edge1] (f5) -- (v5) (f5) -- (v6) (f5) -- (v7) (f5) -- (v8);
    \draw[edge1] (f1) -- (v1) (f1) -- (v2) (f1) -- (v5) (f1) -- (v6);
    \draw[edge1] (f2) -- (v2) (f2) -- (v4) (f2) -- (v6) (f2) -- (v8);
    \draw[edge1] (f3) -- (v3) (f3) -- (v4) (f3) -- (v7) (f3) -- (v8);
    \draw[edge1] (f4) -- (v1) (f4) -- (v3) (f4) -- (v5) (f4) -- (v7);
    \draw[edge1] (f5) -- (v5) (f5) -- (v6) (f5) -- (v7) (f5) -- (v8);

    \draw[edge] (v1) -- (v2) -- (v4) -- (v3) -- cycle;
    \draw[edge] (v5) -- (v6) -- (v8) -- (v7) -- cycle;
    \draw[edge] (v1) -- (v5) (v2) -- (v6) (v4) -- (v8) (v3) -- (v7);
  \end{tikzpicture}
  \caption{The barycentric subdivision of the plane graph of a cube. Edges of type 0 are \iftoggle{bw}{dotted}{red}, edges of type 1 are \iftoggle{bw}{dashed}{green} and edges of type 2 are \iftoggle{bw}{solid}{black}.}\label{fig:barycentric}
\end{figure}

We use the drawing conventions from \cref{fig:barycentric} for the types of the edges in all figures. Since the vertex types can be deduced from the edge types, we do not display them in the figures.

\begin{definition}\label{def:decoration}
  A \emph{decoration}~$D$ is a 2\dash connected plane graph with vertex set $V$ and edge set $E$, together with a labeling function $t\colon V \cup E \to \{0,1,2\}$, and an outer face which contains vertices~$v_0, v_1, v_2$, such that
  \begin{enumerate}
    \item all inner faces are triangles
    \item for each edge~$e = (v, w)$,~$\{t(e), t(v), t(w)\} = \{0, 1, 2\}$
    \item for each vertex~$v$ with~$t(v) = i$, the types of incident edges are~$j$ and~$k$ with~$\{i, j, k\} = \{0, 1, 2\}$. Two consecutive edges with an inner face in between can not have the same type.
    \item for each inner vertex~$v$
      \begin{align}
        t(v) = 1 \quad&\Rightarrow\quad \operatorname{deg}(v) = 4 \\
        t(v) \neq 1 \quad&\Rightarrow\quad \operatorname{deg}(v) > 4
      \end{align}
      for each vertex~$v$ in the outer face and different from $v_0, v_1, v_2$
      \begin{align}
        t(v) = 1 \quad&\Rightarrow\quad \operatorname{deg}(v) = 3 \\
        t(v) \neq 1 \quad&\Rightarrow\quad \operatorname{deg}(v) > 3
      \end{align}
      and
      \begin{align}
        &t(v_0), t(v_2) \neq 1 \\
        &t(v_1) = 1 \quad\Rightarrow\quad \operatorname{deg}(v_1) = 2 \\
        &t(v_1) \neq 1 \quad\Rightarrow\quad \operatorname{deg}(v_1) > 2
      \end{align}
  \end{enumerate}
\end{definition}

Note that condition 3 implies that all inner vertices have an even degree.

For all $\{i, j, k\} = \{0, 1, 2\}$, the $k$-side of a decoration $D$ is the path on the border of the outer face between $v_i$ and $v_j$ that does not pass through $v_k$.

We can fill each triangular face of a chamber system $C_G$ with a decoration, by identifying the vertex of type $i$ with $v_i$ for $i \in \{0,1,2\}$ and identifying corresponding vertices on the boundary. This results in a new chamber system $C_{G'}$ of a new graph $G'$, as can be seen in \cref{fig:decorate}.

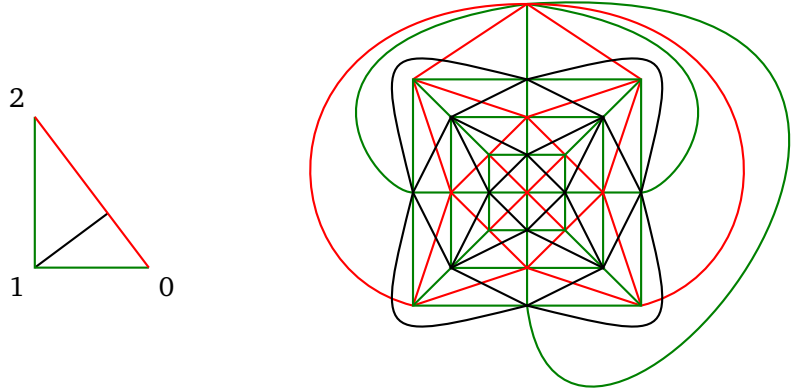
\begin{figure}[htp]
  \centering
  \parbox[c]{0.25\textwidth}{
    \centering
    \begin{tikzpicture}[scale=0.5]
      \coordinate (v0) at (3, 0);
      \node[below right] at (v0) {0};
      \coordinate (v1) at (0, 0);
      \node[below left] at (v1) {1};
      \coordinate (v2) at (0, 4);
      \node[above left] at (v2) {2};

      \draw[edge] (v1) -- ($(v0)!(v1)!(v2)$);
      \draw[edge1] (v0) -- (v1) -- (v2);
      \draw[edge0] (v0) -- (v2);
    \end{tikzpicture}
  }
  \parbox[c]{0.7\textwidth}{
    \begin{tikzpicture}[scale=0.5]
      \useasboundingbox (-5,-2.5) rectangle (11,8.5);

      \coordinate (v1) at (0,6);
      \coordinate (v2) at (6,6);
      \coordinate (v3) at (0,0);
      \coordinate (v4) at (6,0);
      \coordinate (v5) at (2,4);
      \coordinate (v6) at (4,4);
      \coordinate (v7) at (2,2);
      \coordinate (v8) at (4,2);

      \coordinate (e1) at (3,6);
      \coordinate (e2) at (6,3);
      \coordinate (e3) at (3,0);
      \coordinate (e4) at (0,3);
      \coordinate (e5) at (3,4);
      \coordinate (e6) at (4,3);
      \coordinate (e7) at (3,2);
      \coordinate (e8) at (2,3);
      \coordinate (e9) at (1,5);
      \coordinate (e10) at (5,5);
      \coordinate (e11) at (5,1);
      \coordinate (e12) at (1,1);

      \coordinate (f0) at (3,8);
      \coordinate (f1) at (3,5);
      \coordinate (f2) at (5,3);
      \coordinate (f3) at (3,1);
      \coordinate (f4) at (1,3);
      \coordinate (f5) at (3,3);

      \draw[edge1] (f0) -- (e1) -- (f1) -- (e5) -- (f5);
      \draw[edge1] (f0) .. controls (10,7) and (7,3) .. (e2) -- (f2) -- (e6) -- (f5);
      \draw[edge1] (f0) .. controls (-4,7) and (-1,3) .. (e4) -- (f4) -- (e8) -- (f5);
      \draw[edge1] (f0) .. controls (18,9) and (4,-8) .. (e3) -- (f3) -- (e7) -- (f5);
      \draw[edge1] (f1) -- (e10) -- (f2) -- (e11) -- (f3) -- (e12) -- (f4) -- (e9) -- (f1);

      \draw[edge0] (f0) -- (v1) (f0) -- (v2)
                   (f0) .. controls (-4,8) and (-4,1) .. (v3)
                   (f0) .. controls (10,8) and (10,1) .. (v4);
      \draw[edge0] (f1) -- (v1) (f1) -- (v2) (f1) -- (v5) (f1) -- (v6);
      \draw[edge0] (f2) -- (v2) (f2) -- (v4) (f2) -- (v6) (f2) -- (v8);
      \draw[edge0] (f3) -- (v3) (f3) -- (v4) (f3) -- (v7) (f3) -- (v8);
      \draw[edge0] (f4) -- (v1) (f4) -- (v3) (f4) -- (v5) (f4) -- (v7);
      \draw[edge0] (f5) -- (v5) (f5) -- (v6) (f5) -- (v7) (f5) -- (v8);

      \draw[edge1] (v1) -- (v2) -- (v4) -- (v3) -- cycle;
      \draw[edge1] (v5) -- (v6) -- (v8) -- (v7) -- cycle;
      \draw[edge1] (v1) -- (v5) (v2) -- (v6) (v4) -- (v8) (v3) -- (v7);

      \draw[edge] (e1) -- (e10) -- (e5) -- (e9) -- cycle;
      \draw[edge] (e2) -- (e11) -- (e6) -- (e10) -- cycle;
      \draw[edge] (e3) -- (e12) -- (e7) -- (e11) -- cycle;
      \draw[edge] (e4) -- (e9) -- (e8) -- (e12) -- cycle;
      \draw[edge] (e5) -- (e6) -- (e7) -- (e8) -- cycle;
      \draw[edge] (e1) .. controls (7,7) .. (e2)
                       .. controls (7,-1) .. (e3)
                       .. controls (-1,-1) .. (e4)
                       .. controls (-1,7) .. cycle;
    \end{tikzpicture}
  }
  \caption{The decoration \emph{ambo} applied to the cube of \cref{fig:barycentric}. The resulting graph $G'$ is the one in \iftoggle{bw}{solid lines}{black}.}\label{fig:decorate}
\end{figure}

This is very similar to the lsp operations of \cite{Brinkmann2017}. We are constructing graphs by subdividing the chambers of the chamber system. One key difference is that we impose no restrictions on the connectivity. This means that we can apply decorations to arbitrary embedded graphs, but when applied to a polyhedron -- \ie a 3\dash connected plane graph -- it is possible that the result has a lower connectivity. We will address this problem later with additional restrictions on decorations.

For now, we will repeat Definition 5.1 of \cite{Brinkmann2017} without the restrictions on the connectivity.

\begin{definition}\label{def:lsp}
Let $T$ be a connected periodic tiling of the Euclidean plane with chamber system $C_T$, that is given by a barycentric subdivision that is invariant under the symmetries of $T$.
Let $v_0,v_1,v_2$ be points in the Euclidean plane so that for $0\le i < j \le 2$ the line $L_{i,j}$ through $v_i$ and $v_j$ is a mirror axis of the tiling.

If the angle between $L_{0,1}$ and $L_{2,1}$ is $90$ degrees, the angle between $L_{2,1}$ and $L_{2,0}$ is $30$ degrees and consequently the angle between $L_{0,1}$ and $L_{0,2}$ is $60$ degrees, then the triangle $v_0,v_1,v_2$ subdivided into chambers as given by $C_T$ and the corners $v_0,v_1,v_2$ labelled with their names $v_0,v_1,v_2$ is called a \emph{local symmetry-preserving operation}, {\em lsp operation} for short.

The result $O(G)$ of applying an lsp operation $O$ to a connected graph $G$ is given by subdividing each chamber $C$ of the chamber system $C_G$ with $O$ by identifying for $0\le i\le 2$ the vertices of $O$ labelled $v_i$ with the vertices labelled $i$ in $C$.
\end{definition}

An lsp operation is called \emph{$k$\dash connected} for $k \in \{1,2,3\}$ if it is derived from a $k$\dash connected tiling $T$. So the original definition was for 3\dash connected lsp operations only. In order to correctly determine the connectivity, we first need to identify which chamber systems correspond to $k$\dash connected graphs. To decide whether a graph $G$ is $k$\dash connected based on its chamber system $C_G$, we can look at the \emph{type-1 cycles} in $C_G$. A type-1 cycle is a cycle in the subgraph of $C_G$ that consists of the type-1 edges only. A type-1 cycle is empty if there are no vertices on the inside or on the outside of the cycle in this type-1 subgraph. Note that in the graph $C_G$ these cycles are not necessarily empty.

\begin{lemma}\label{lemma:connectivity}
  A plane graph $G$ is
  \begin{enumerate}
    \item 2\dash connected if and only if $C_G$ contains no type-1 cycles of length 2.
    \item 3\dash connected if and only if $G$ is 2\dash connected and $C_G$ contains no non-empty type-1 cycles of length 4.
  \end{enumerate}
\end{lemma}
\begin{proof}
  \begin{enumerate}
    \item Suppose $C_G$ contains a type-1 cycle of length 2. This cycle contains one type-0 vertex $v$, incident to at least one type-2 edge inside the cycle and at least one type-2 edge outside the cycle (see \cref{fig:cut-vertex}), because $C_G$ is a barycentric subdivision. It is clear that $v$ has to be a cut-vertex of $G$.

    Conversely, if $G$ has a cut-vertex $v$, there is a face of $G$ for which $v$ occurs at least two times in its border. In $C_G$ this face corresponds with a type-2 vertex, incident with at least two type-1 edges to $v$. These edges form a type-1 cycle in $C_G$.

    \item Suppose $C_G$ contains a non-empty type-1 cycle of length 4, as can be seen in \cref{fig:2-cut}. This cycle contains two type-0 vertices $v$ and $w$, with incident type-2 edges at both sides of the cycle. Removing $v$ and $w$ from $G$ results in a disconnected graph.

    If $G$ is 2\dash connected but not 3\dash connected, there are two vertices $v$ and $w$ that disconnect $G$ when removed. So there are two non-empty subgraphs of $G$ that are only connected by $v$ and $w$, as in \cref{fig:2-cut}. This means that there is a non-empty type-1 cycle in $C_G$.
  \end{enumerate}

  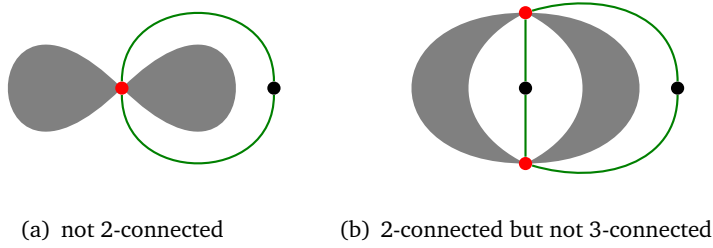
\begin{figure}[htp]
    \centering
    \subbottom[not 2-connected]{
      \begin{tikzpicture}
        \useasboundingbox (-2.5,-1.5) rectangle (2.5,1.5);
        \fill[edge,gray] (0, 0) .. controls (-2,2) and (-2,-2) .. (0, 0);
        \fill[edge,gray] (0, 0) .. controls (2,2) and (2,-2) .. (0, 0);
        \node[vertex0] (v0) at (0, 0) {};
        \node[vertex] (v1) at (2, 0) {};
        \draw[edge1] (v0) .. controls (0,1.3) and (2,1.3) .. (v1)
                     (v0) .. controls (0,-1.3) and (2,-1.3) .. (v1);
      \end{tikzpicture}
      \label{fig:cut-vertex}
    }
    \subbottom[2-connected but not 3-connected]{
      \begin{tikzpicture}
        \useasboundingbox (-2.5,-1.5) rectangle (2.5,1.5);
        \fill[edge,gray] (0, 1) .. controls (-2,1) and (-2,-1) .. (0, -1)
                              .. controls (-1,-0.5) and (-1,0.5) .. (0, 1);
        \fill[edge,gray] (0, 1) .. controls (2,1) and (2,-1) .. (0, -1)
                              .. controls (1,-0.5) and (1,0.5) .. (0, 1);
        \node[vertex] (v0) at (0, 0) {};
        \node[vertex] (v1) at (2, 0) {};
        \node[vertex0] (v2) at (0, 1) {};
        \node[vertex0] (v3) at (0, -1) {};
        \draw[edge1] (v0) -- (v2) .. controls (1,1.3) and (2,1) .. (v1)
                     .. controls (2,-1) and (1,-1.3) .. (v3) -- (v0);
      \end{tikzpicture}
      \label{fig:2-cut}
    }
    \caption{Two graphs with type-1 cycles. The gray area contains the graph. Only the type-1 edges of the chamber system are shown. The type-0 vertices are \iftoggle{bw}{triangles}{red} and the type-2 vertices are \iftoggle{bw}{cirkels}{black}.}\label{fig:cuts}
  \end{figure}
\end{proof}

Note that this theorem only holds for plane graphs, since the proof relies on the Jordan curve theorem. A counterexample to an equivalent theorem for embedded graphs of higher genus is the dual of a 3\dash connected graph on the torus, which can have a 2\dash cut (see \cite{Bokal2018}).

Since we introduced a more general definition of lsp operations, we can also formulate a more general version of Theorem 5.2 in \cite{Brinkmann2017}.

\begin{theorem}
  If $G$ is a $k$\dash connected plane graph with $k \in \{1,2,3\}$, and $O$ is a $k$\dash connected lsp operation, then $O(G)$ is a $k$\dash connected plane graph.
\end{theorem}
\begin{proof}
  It is clear that $O(G)$ is a plane graph. For $k = 1$, we know that $T$ and $G$ are connected, and it follows easily that $O(G)$ is connected. For $k = 3$, the proof is given in~\cite{Brinkmann2017}. For $k = 2$, we will prove that there is no cut-vertex in $O(G)$.

  A type-1 cycle of length 2 in $C_{O(G)}$ is either completely contained in one chamber of $C_G$\footnote{With a chamber of $C_G$ in $C_{O(G)}$, we mean the area that was a chamber of $C_G$ before it was subdivided by $O$} (see \cref{fig:2cycle1}), or it is split between two chambers of $C_G$ (see \cref{fig:2cycle2}). Both cases cannot appear, as for any chamber (\resp any pair of adjacent chambers) there is an isomorphism between this chamber (\resp these two chambers) and the corresponding area in $T$, and according to \cref{lemma:connectivity} $T$ has no type-1 cycles of length 2.

  This implies that $C_{O(G)}$ contains no type-1 cycles of length 2, and thus, invoking once again \cref{lemma:connectivity}, $O(G)$ contains no cut-vertices.
\end{proof}

\begin{figure}[htp]
  \centering
  \setlength{\tabcolsep}{0.5cm}
  \loosesubcaptions
  \begin{tabular}[c]{cc}
    \subbottom[]{
      \begin{tikzpicture}[scale=0.5]
        \coordinate (v0) at (3, 0);
        \coordinate (v1) at (0, 0);
        \coordinate (v2) at (0, 4);
        \draw[edge] (v0) -- (v1) -- (v2) -- cycle;
        \draw[edge1] (0.5,1.75) to[out=0,in=90] (1.5,0.75) to[out=180,in=-90] cycle;
      \end{tikzpicture}
      \label{fig:2cycle1}
    } &
    \subbottom[]{
      \begin{tikzpicture}[scale=0.5]
        \coordinate (v0) at (3, 0);
        \coordinate (v1) at (0, 0);
        \coordinate (v2) at (0, 4);
        \coordinate (v0') at (-3, 0);
        \draw[edge] (v1) -- (v2) -- (v0) -- (v0') -- (v2);
        \draw[edge1] (0,2.25) to[out=-30,in=30] (0,0.75) to[out=150,in=-150] cycle;
      \end{tikzpicture}
      \label{fig:2cycle2}
    }
  \end{tabular}
  \caption{The different situations where type-1 cycles of length 2 can occur.}\label{fig:2cycles}
\end{figure}
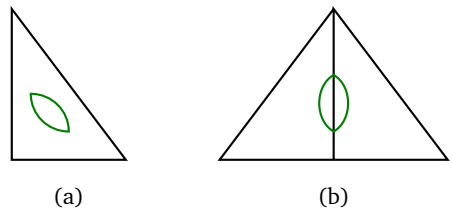

We can prove similar properties for decorations, but it is easier to use the correspondence between lsp operations and decorations. Although the way they are defined is rather different, in reality they are the same thing. The triangle $v_0,v_1,v_2$ of an lsp operation that is derived from a tiling has exactly the properties of a decoration, and each decoration can be derived as an lsp operation from a tiling.

\begin{theorem}\label{thm:lsp}
  Each decoration defines an lsp operation and vice versa.
\end{theorem}
\begin{proof}
  It is straightforward that the graph defined by an lsp operation is unique and satisfies the conditions of \cref{def:decoration}. We still have to prove that each decoration defines an lsp operation.

  Given a decoration $D$, we can take the hexagonal lattice $H$ and use $D$ to decorate each chamber of the chamber system $C_H$. The result will be a chamber system $C_T$ of a tiling $T$.

  We will first prove that the type-2 subgraph of $D$ is connected, by induction on the number of triangles. There is always at least one triangle in $D$ that shares one or two edges with the outer face. We remove these edges, and call the result $D'$. It is clear that $D'$ still satisfies properties 1-3 of \cref{def:decoration}, and by induction its type-2 subgraph is connected. If one of the removed edges has type $2$, it is connected to $D'$ by a vertex of type 0 or 1 with degree at least 3, and therefore it is connected to the type-2 subgraph of $D'$.

  Given vertices $u$ and $v$ in the type-2 subgraph of $C_T$, there exists a sequence of chambers $C_0, \dotsc, C_n$ of $H$ such that two consecutive chambers $C_i$ and $C_{i + 1}$ share one side, and $u$ is contained in $C_0$ and $v$ in $C_n$. Since there are at least two vertices on each side of $D$, and they are not both of type 2, at least one of them is in the type-2 subgraph of $C_T$. Thus, there is a type-2 path between $u$ and $v$ that passes through all chambers in the sequence $C_0, \dotsc, C_n$, and the type-2 subgraph of $C_T$ is connected. It follows immediately that $T$ is connected too.

  We can choose the vertices of one chamber of $C_H$ in $T$ as $v_0$, $v_1$ and $v_2$. This satisfies the properties of \cref{def:lsp}, and it is clear that the decoration defined by the triangle $v_0, v_1, v_2$ is equal to $D$.
\end{proof}

This correspondence can be further extended to 2\dash connected and 3\dash connected operations.

\begin{definition}\label{def:2-decoration}
  A \emph{2\dash connected decoration} is a decoration with
  \begin{enumerate}
    \item no type-1 cycles of length 2
    \item no internal type-1 edges between two vertices on a single side
  \end{enumerate}
\end{definition}

\begin{definition}\label{def:3-decoration}
  A \emph{3\dash connected decoration} is a 2\dash connected decoration with
  \begin{enumerate}
    \item no type-1 edge between sides 0 and 2
    \item no non-empty type-1 cycles of length 4
  \end{enumerate}
\end{definition}

Note that, when seen as a graph, a decoration is always at least 2\dash connected.

\begin{theorem}\label{thm:lsp-2}
  Each 2\dash connected decoration $D$ defines a 2\dash connected lsp operation and vice versa.
\end{theorem}
\begin{proof}
  A 2\dash connected decoration is a decoration, so it follows from \cref{thm:lsp} that $D$ defines an lsp operation. We still have to prove that the corresponding tiling~$T$ is 2\dash connected. If $T$ is not 2\dash connected, there is a type-1 cycle of length 2 in $C_T$. If this cycle is completely contained in the triangle $v_0,v_1,v_2$, there is a cycle of length 2 in $D$ too, which is impossible. The only other possibility is that the cycle of length 2 is cut in half by $L_{ij}$, but then there would be an internal type-1 edge between 2 vertices on $L_{ij}$, which is a side of $D$.

  A 2\dash connected lsp operation with corresponding tiling $T$ defines a decoration $D$ according to \cref{thm:lsp}. We still have to prove that the extra conditions of \cref{def:2-decoration} are satisfied. If there is a type-1 cycle of length 2 in $D$, this cycle occurs in $C_T$ too, and $T$ would not be 2\dash connected. If there is an internal type-1 edge between 2 vertices on the same side, this will result in a cycle of length 2 in $T$ because this side lies on a mirror axis of $T$.
\end{proof}

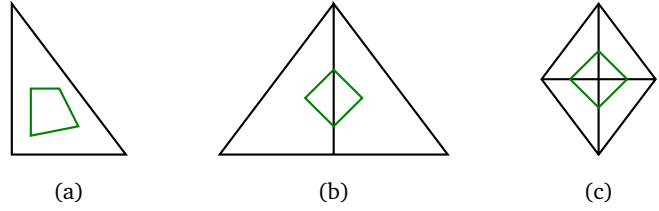
\begin{figure}[htp]
  \centering
  \setlength{\tabcolsep}{0.5cm}
  \loosesubcaptions
  \begin{tabular}[c]{ccc}
    \subbottom[]{
      \begin{tikzpicture}[scale=0.5]
        \coordinate (v0) at (3, 0);
        \coordinate (v1) at (0, 0);
        \coordinate (v2) at (0, 4);
        \draw[edge] (v0) -- (v1) -- (v2) -- cycle;
        \draw[edge1] (0.5,1.75) -- (1.25,1.75) -- (1.75,0.75) -- (0.5,0.5) -- cycle;
      \end{tikzpicture}
      \label{fig:4cycle1}
    } &
    \subbottom[]{
      \begin{tikzpicture}[scale=0.5]
        \coordinate (v0) at (3, 0);
        \coordinate (v1) at (0, 0);
        \coordinate (v2) at (0, 4);
        \coordinate (v0') at (-3, 0);
        \draw[edge] (v1) -- (v2) -- (v0) -- (v0') -- (v2);
        \draw[edge1] (0,2.25) -- (0.75,1.5) -- (0,0.75) -- (-0.75,1.5) -- cycle;
      \end{tikzpicture}
      \label{fig:4cycle2}
    } &
    \subbottom[]{
      \begin{tikzpicture}[scale=0.25]
        \coordinate (v0) at (3, 0);
        \coordinate (v1) at (0, 0);
        \coordinate (v2) at (0, 4);
        \coordinate (v0') at (-3, 0);
        \coordinate (v2') at (0, -4);
        \draw[edge] (v2) -- (v0) -- (v2') -- (v0') -- cycle (v0) -- (v0') (v2) -- (v2');
        \draw[edge1] (0,1.5) -- (1.5,0) -- (0,-1.5) -- (-1.5,0) -- cycle;
      \end{tikzpicture}
      \label{fig:4cycle4}
    }
  \end{tabular}
  \caption{The different situations where non-empty type-1 cycles of length 4 can occur.}\label{fig:4cycles}
\end{figure}

\begin{theorem}\label{thm:lsp-3}
  Each 3\dash connected decoration $D$ defines a 3\dash connected lsp operation and vice versa.
\end{theorem}
\begin{proof}
  A 3\dash connected decoration defines a 2\dash connected lsp operation. If $T$ is not 3\dash connected, there is a non-empty type-1 cycle of length 4. If this cycle is completely contained in the triangle $v_0,v_1,v_2$, there is a type-1 cycle of length 4 in $D$. If the cycle is cut in half by $L_{ij}$, there is an internal type-1 path of length 2 between 2 vertices on $L_{ij}$, which is a side of $D$. If the cycle is cut in four, as in \cref{fig:4cycle4}, there is a type-1 edge between sides 0 and 2.

  A 3\dash connected lsp operation with corresponding tiling $T$ defines a 2\dash connected decoration $D$. If there is a type-1 cycle of length 4 in $D$, this cycle occurs in $C_T$ too, and $T$ would not be 3\dash connected. If there is an internal type-1 path of length 2 between 2 vertices on the same side, or a type-1 edge between sides 0 and 2, this will result in a cycle of length 4 in $T$.
\end{proof}

\section{Predecorations}

The generation of all decorations will be split into two phases. In the first phase, we will construct the type-1 subgraph, consisting of all edges of type 1.

Let $n_A$ be the number of vertices in the type-1 subgraph of degree 1 with a neighbouring vertex of degree 2, $n_B$ the number of remaining vertices of degree 1, and $n_C$ the number of quadrangles with three vertices of degree 2.

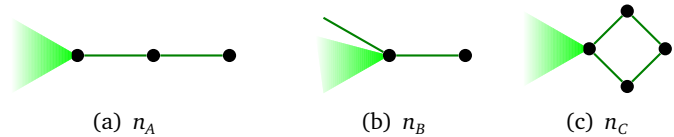
\begin{figure}[htp]
  \centering
  \subbottom[$n_A$]{
    \label{fig:nA}
    \begin{tikzpicture}
      \node[vertex] (v0) at (0,0) {};
      \node[vertex] (v1) at (1,0) {};
      \node[vertex] (v2) at (2,0) {};
      \draw[edge1] (v0) -- (v1) -- (v2);
      \shade[left color=white,right color=\iftoggle{bw}{black}{green}] (v0) -- +(150:1) -- +(210:1) -- (v0);
    \end{tikzpicture}
  }
  \qquad
  \subbottom[$n_B$]{
    \label{fig:nB}
    \begin{tikzpicture}
      \node[vertex] (v0) at (0,0) {};
      \node[vertex] (v1) at (1,0) {};
      \draw[edge1] (v0) -- (v1) (v0) -- +(150:1);
      \shade[left color=white,right color=\iftoggle{bw}{black}{green}] (v0) -- +(165:1) -- +(210:1) -- (v0);
    \end{tikzpicture}
  }
  \quad
  \subbottom[$n_C$]{
    \label{fig:nC}
    \begin{tikzpicture}
      \node[vertex] (v0) at (0,0) {};
      \node[vertex] (v1) at (0.5,0.5) {};
      \node[vertex] (v2) at (1,0) {};
      \node[vertex] (v3) at (0.5,-0.5) {};
      \draw[edge1] (v0) -- (v1) -- (v2) -- (v3) -- (v0);
      \shade[left color=white,right color=\iftoggle{bw}{black}{green}] (v0) -- +(150:1) -- +(210:1) -- (v0);
    \end{tikzpicture}
  }
  \caption{The subgraphs counted as $n_A$, $n_B$ and $n_C$}\label{fig:nA12}
\end{figure}

\newpage

\begin{lemma}\label{lemma:predecorations}
  Let $D$ be a decoration. The type-1 subgraph $D_1$ of $D$ has the following properties:
  \begin{enumerate}
    \item all inner faces are quadrangles;
    \item each inner vertex has degree at least 3;
    \item $n_A \leq 2$ and $n_A + n_B + n_C \leq 3$.
  \end{enumerate}
\end{lemma}
\begin{proof}
  It follows immediately from the properties of a decoration (\cref{def:decoration}) that the inner faces of $D_1$ are quadrangles and the inner vertices have degree at least 3.

  Each area bounded by a quadrangle in $D_1$ contains one vertex of type 1 in $D$. The only other difference between $D$ and $D_1$ is in the outer face of $D_1$, where type-1 vertices of degree 3 in $D$ (a 3-completion), and at most one of degree 2 in $D$ (a 2-completion), can be present in $D$. If there is a type-1 vertex of degree 2, then that vertex is $v_1$. An example can be seen in \cref{fig:completion}.

  The subgraph in \cref{fig:nC} can only occur if the rightmost vertex $v$ of degree two is $v_0$, $v_1$ or $v_2$, or if $v_1$ is a type-1 vertex of degree 2 connected to this vertex. Each of the three vertices of degree 2 in this subgraph of $D_1$ corresponds to $v_0$, $v_1$, $v_2$ or a vertex of degree at least 4 in $D$. The inner edges of the quadrangle in $D$ contribute exactly one to the degree of these vertices. This implies that either there is a 2-completion here (in which case $v_1$ is connected to $v$), or there are two 3-completions which do not involve $v$ (in which case $v$ is $v_0$, $v_1$ or $v_2$).

  The subgraph in \cref{fig:nB} can only occur if the rightmost vertex $v$ is $v_0$, $v_1$ or $v_2$. This vertex of degree 1 in $D_1$ corresponds to a vertex of degree at most 3 in $D$, which is only possible in $v_0$, $v_1$ or $v_2$.

  The subgraph in \cref{fig:nA} can only occur if the rightmost vertex $v$ is $v_0$ or $v_2$. There are two neighbouring cut-vertices of $D_1$ in this subgraph, which do not correspond to cut-vertices in $D$. This is only possible if both of these vertices are the middle vertex of a 3-completion. This increases the degree of $v$ in $D$ to 2, which is only possible in $v_0$ or $v_2$. The degree of $v$ can be 3 if there is a 2-completion too, but then $v_1$ is contained in this 2-completion and $v$ still has to be $v_0$ or $v_2$.

  We find that $n_A \leq \abs{\{v_0, v_2\}} = 2$ and $n_A + n_B + n_C \leq \abs{\{v_0, v_1, v_2\}} = 3$.
\end{proof}

\begin{definition}\label{def:predecoration}
  A \emph{predecoration} is a connected plane graph with an outer face that satisfies the properties of \cref{lemma:predecorations}.
\end{definition}

Given a predecoration $P$, we can try to add edges, vertices and labels to get a decoration with $P$ as its type-1 subgraph. We will have to add one type-1 vertex in each inner face of $P$, as in \cref{fig:completion}. Then we can add type-1 vertices in the outer face, and connect them to three consecutive vertices of $P$. Finally, we can add a type-1 vertex in the outer face and connect it to two consecutive vertices of $P$. This vertex has to be $v_1$.

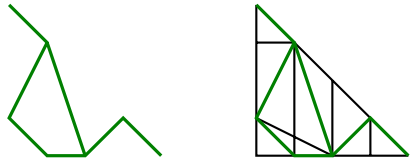
\begin{figure}[htp]
  \centering
  \begin{tikzpicture}[scale=0.5]
    \coordinate (1) at (0,4);
    \coordinate (2) at (1,3);
    \coordinate (3) at (0,1);
    \coordinate (4) at (1,0);
    \coordinate (5) at (2,0);
    \coordinate (6) at (3,1);
    \coordinate (7) at (4,0);

    \draw[edge1,very thick] (1) -- (2) -- (3) -- (4) -- (5) -- (6) -- (7) (2) -- (5);
  \end{tikzpicture}
  \hspace{1cm}
  \begin{tikzpicture}[scale=0.5]
    \coordinate (1) at (0,4);
    \coordinate (2) at (1,3);
    \coordinate (3) at (0,1);
    \coordinate (4) at (1,0);
    \coordinate (5) at (2,0);
    \coordinate (6) at (3,1);
    \coordinate (7) at (4,0);

    \draw[edge,thick] (1) |- (4) (0,3) -- (2) -- (4) (3) -- (5) -- (2,2) (2) -- (6) -- (3,0) (5) -- (7);
    \draw[edge1,very thick] (1) -- (2) -- (3) -- (4) -- (5) -- (6) -- (7) (2) -- (5);
  \end{tikzpicture}
  \caption{A predecoration with a possible completion. The edges of type 0 and 2 are both shown in black.}\label{fig:completion}
\end{figure}

By definition, the type-1 subgraph of a decoration $D$ is a predecoration. Unfortunately, not each predecoration corresponds to a type-1 subgraph of some decoration. This is \eg the case if there are too many cut-vertices, as in \cref{fig:uncompletable}.

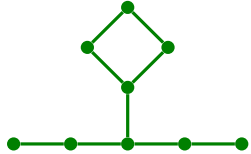
\begin{figure}[htp]
  \centering
  \begin{tikzpicture}[scale=0.75]
    \node[vertex1] (0) at (0,0) {};
    \node[vertex1] (1) at (135:1) {};
    \node[vertex1] (2) at (45:1) {};
    \node[vertex1] (3) at ($(1) + (2)$) {};
    \node[vertex1] (4) at (-2,-1) {};
    \node[vertex1] (5) at (-1,-1) {};
    \node[vertex1] (6) at (0,-1) {};
    \node[vertex1] (7) at (1,-1) {};
    \node[vertex1] (8) at (2,-1) {};

    \draw[edge1,very thick] (0) -- (1) -- (3) -- (2) -- (0) -- (6)
                            (4) -- (5) -- (6) -- (7) -- (8);
  \end{tikzpicture}
  \caption{A predecoration that cannot be completed.}\label{fig:uncompletable}
\end{figure}

\section{Construction of predecorations}

All predecorations can be constructed from the base decorations $K_2$ and $C_4$ (see \cref{fig:base}) using the 10 extension operations shown in \cref{fig:extensions}. We will prove this by showing that each predecoration, with the exception of $K_2$ and $C_4$, can be reduced by the inverse of one of the extension operations. We will then use the \emph{canonical construction path} method \cite{McKay1998} to generate all predecorations without isomorphic copies.

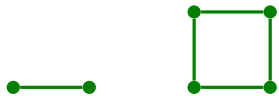
\begin{figure}[htp]
  \centering
  \begin{tikzpicture}
    \node[vertex1] (0) at (0,0) {};
    \node[vertex1] (1) at (1,0) {};
    \draw[edge1,very thick] (0) -- (1);
  \end{tikzpicture}
  \hspace{1cm}
  \begin{tikzpicture}
    \node[vertex1] (0) at (0,0) {};
    \node[vertex1] (1) at (1,0) {};
    \node[vertex1] (2) at (0,1) {};
    \node[vertex1] (3) at (1,1) {};
    \draw[edge1,very thick] (0) -- (1) -- (3) -- (2) -- (0);
  \end{tikzpicture}
  \caption{The base predecorations.}\label{fig:base}
\end{figure}

\begin{figure}[htp]
  \centering
  \begin{tabular}{rl@{\hskip 1cm}rl@{\hskip 1cm}rl}
    \toprule \addlinespace
    &
    \begin{tikzpicture}[scale=0.75,baseline=-0.75ex]
      \node[vertex0] (0) at (0,0) {};
      \shade[left color=white,right color=black] (0) -- +(150:1) -- +(210:1) -- (0);
      \shade[left color=black,right color=white] (0) -- +(30:1) -- +(-30:1) -- (0);
    \end{tikzpicture} &&
    \begin{tikzpicture}[scale=0.75,baseline=-0.75ex]
      \node[vertex] (0) at (0,0) {};
      \shade[left color=white,right color=black] (0) -- +(150:1) -- +(210:1) -- (0);
    \end{tikzpicture} &
    \multicolumn{2}{c}{
      \begin{tikzpicture}[scale=0.75,baseline=1ex]
        \shade[left color=white,right color=black] (0,1) -- ++(150:0.5) -- +(0,-1.5) -- (0,0) -- (0,1);
        \node[vertex] (0) at (0,0) {};
        \node[vertex] (2) at (0,1) {};
        \draw[edge,very thick] (0) -- (2);
      \end{tikzpicture}
      \begin{tikzpicture}[scale=0.75,baseline=1ex]
        \shade[left color=white,right color=black] (0,1) -- ++(150:0.5) -- (-0.443,-0.443) -- (0,0) -- (0,1);
        \shade[top color=black,bottom color=white] (1,0) -- ++(-60:0.5) -- (-0.443,-0.443) -- (0,0) -- (1,0);
        \node[vertex] (0) at (0,0) {};
        \node[vertex] (1) at (1,0) {};
        \node[vertex] (2) at (0,1) {};
        \draw[edge,very thick] (2) -- (0) -- (1);
      \end{tikzpicture}
      \begin{tikzpicture}[scale=0.75,baseline=1ex]
        \shade[left color=white,right color=black] (0,1) -- ++(150:0.5) -- (-0.443,-0.443) -- (0,0) -- (0,1);
        \shade[top color=black,bottom color=white] (1,0) -- (1.443,-0.443) -- (-0.443,-0.443) -- (0,0) -- (1,0);
        \shade[left color=black,right color=white] (1,1) -- ++(30:0.5) -- (1.443,-0.443) -- (1,0) -- (1,1);
        \node[vertex] (0) at (0,0) {};
        \node[vertex] (1) at (1,0) {};
        \node[vertex] (2) at (0,1) {};
        \node[vertex] (3) at (1,1) {};
        \draw[edge,very thick] (2) -- (0) -- (1) -- (3);
      \end{tikzpicture}
    } \\ \addlinespace
    \midrule \addlinespace
    1. &
    \begin{tikzpicture}[scale=0.75,baseline=-0.75ex]
      \node[vertex0] (0) at (0,0) {};
      \node[vertex0] (1) at (1,0) {};
      \shade[left color=white,right color=black] (0) -- +(150:1) -- +(210:1) -- (0);
      \shade[left color=black,right color=white] (1) -- +(30:1) -- +(-30:1) -- (1);
      \draw[edge1,very thick] (0) -- (1);
    \end{tikzpicture} &
    2. &
    \begin{tikzpicture}[scale=0.75,baseline=-0.75ex]
      \node[vertex] (0) at (0,0) {};
      \node[vertex1] (1) at (1,0) {};
      \shade[left color=white,right color=black] (0) -- +(150:1) -- +(210:1) -- (0);
      \draw[edge1,very thick] (0) -- (1);
    \end{tikzpicture} &
    8. &
    \begin{tikzpicture}[scale=0.75,baseline=1ex]
      \shade[left color=white,right color=black] (0,1) -- ++(150:0.5) -- +(0,-1.5) -- (0,0) -- (0,1);
      \node[vertex] (0) at (0,0) {};
      \node[vertex1] (1) at (1,0) {};
      \node[vertex] (2) at (0,1) {};
      \node[vertex1] (3) at (1,1) {};
      \draw[edge,very thick] (0) -- (2);
      \draw[edge1,very thick] (2) -- (3) -- (1) -- (0);
    \end{tikzpicture} \\ \addlinespace
    3. &
    \begin{tikzpicture}[scale=0.75,baseline=-0.75ex]
      \node[vertex0] (0) at (0,0) {};
      \node[vertex0] (1) at (1,0) {};
      \node[vertex1] (2) at (0.5,0.5) {};
      \node[vertex1] (3) at (0.5,-0.5) {};
      \shade[left color=white,right color=black] (0) -- +(150:1) -- +(210:1) -- (0);
      \shade[left color=black,right color=white] (1) -- +(30:1) -- +(-30:1) -- (1);
      \draw[edge1,very thick] (0) -- (2) -- (1) -- (3) -- (0);
    \end{tikzpicture} &
    5. &
    \begin{tikzpicture}[scale=0.75,baseline=-0.75ex]
      \node[vertex] (0) at (0,0) {};
      \node[vertex1] (1) at (1,0) {};
      \node[vertex1] (2) at (0.5,0.5) {};
      \node[vertex1] (3) at (0.5,-0.5) {};
      \shade[left color=white,right color=black] (0) -- +(150:1) -- +(210:1) -- (0);
      \draw[edge1,very thick] (0) -- (2) -- (1) -- (3) -- (0);
    \end{tikzpicture} &
    9. &
    \begin{tikzpicture}[scale=0.75,baseline=1ex]
      \shade[left color=white,right color=black] (0,1) -- ++(150:0.5) -- (-0.443,-0.443) -- (0,0) -- (0,1);
      \shade[top color=black,bottom color=white] (1,0) -- ++(-60:0.5) -- (-0.443,-0.443) -- (0,0) -- (1,0);
      \node[vertex] (0) at (0,0) {};
      \node[vertex] (1) at (1,0) {};
      \node[vertex] (2) at (0,1) {};
      \node[vertex1] (3) at (1,1) {};
      \draw[edge,very thick] (2) -- (0) -- (1);
      \draw[edge1,very thick] (2) -- (3) -- (1);
    \end{tikzpicture} \\ \addlinespace
    4. &
    \begin{tikzpicture}[scale=0.75,baseline=-0.75ex]
      \node[vertex0] (0) at (0,0) {};
      \node[vertex0] (1) at (1,0) {};
      \node[vertex1] (2) at (0,1) {};
      \node[vertex1] (3) at (1,1) {};
      \shade[left color=white,right color=black] (0) -- +(150:1) -- +(210:1) -- (0);
      \shade[left color=black,right color=white] (1) -- +(30:1) -- +(-30:1) -- (1);
      \draw[edge1,very thick] (0) -- (2) -- (3) -- (1) -- (0);
    \end{tikzpicture} &
    6. &
    \begin{tikzpicture}[scale=0.75,baseline=-0.75ex]
      \node[vertex] (0) at (0,0) {};
      \node[vertex1] (1) at (1,0) {};
      \node[vertex1] (2) at (0,1) {};
      \node[vertex1] (3) at (1,1) {};
      \node[vertex1] (4) at (0,-1) {};
      \node[vertex1] (5) at (1,-1) {};
      \shade[left color=white,right color=black] (0) -- +(150:1) -- +(210:1) -- (0);
      \draw[edge1,very thick] (0) -- (2) -- (3) -- (1) -- (0) -- (4) -- (5) -- (1);
    \end{tikzpicture} &
    10. &
    \begin{tikzpicture}[scale=0.75,baseline=1ex]
      \shade[left color=white,right color=black] (0,1) -- ++(150:0.5) -- (-0.443,-0.443) -- (0,0) -- (0,1);
      \shade[top color=black,bottom color=white] (1,0) -- (1.443,-0.443) -- (-0.443,-0.443) -- (0,0) -- (1,0);
      \shade[left color=black,right color=white] (1,1) -- ++(30:0.5) -- (1.443,-0.443) -- (1,0) -- (1,1);
      \node[vertex] (0) at (0,0) {};
      \node[vertex] (1) at (1,0) {};
      \node[vertex] (2) at (0,1) {};
      \node[vertex] (3) at (1,1) {};
      \draw[edge,very thick] (2) -- (0) -- (1) -- (3);
      \draw[edge1,very thick] (2) -- (3);
    \end{tikzpicture} \\ \addlinespace
    &&
    7. &
    \begin{tikzpicture}[scale=0.75,baseline=-0.75ex]
      \node[vertex] (0) at (0,0) {};
      \node[vertex1] (1) at (0.5,0.5) {};
      \node[vertex1] (2) at (0.5,-0.5) {};
      \node[vertex1] (3) at (1,0) {};
      \node[vertex1] (4) at (1,-1) {};
      \node[vertex1] (5) at (1.5,-0.5) {};
      \shade[left color=white,right color=black] (0) -- +(150:1) -- +(210:1) -- (0);
      \draw[edge1,very thick] (0) -- (1) -- (3) -- (5) -- (4) -- (2) -- (0) (2) -- (3);
    \end{tikzpicture} && \\ \addlinespace
    \bottomrule
  \end{tabular}
  \caption{The extensions. In the first row, the subgraphs before the extension is applied are given. New edges and vertices are \iftoggle{bw}{squares}{green}, and vertices that are broken apart in two new vertices are \iftoggle{bw}{diamonds}{red}. The outer face is always on the outside, and shadowed parts contain at least one vertex.}\label{fig:extensions}
\end{figure}
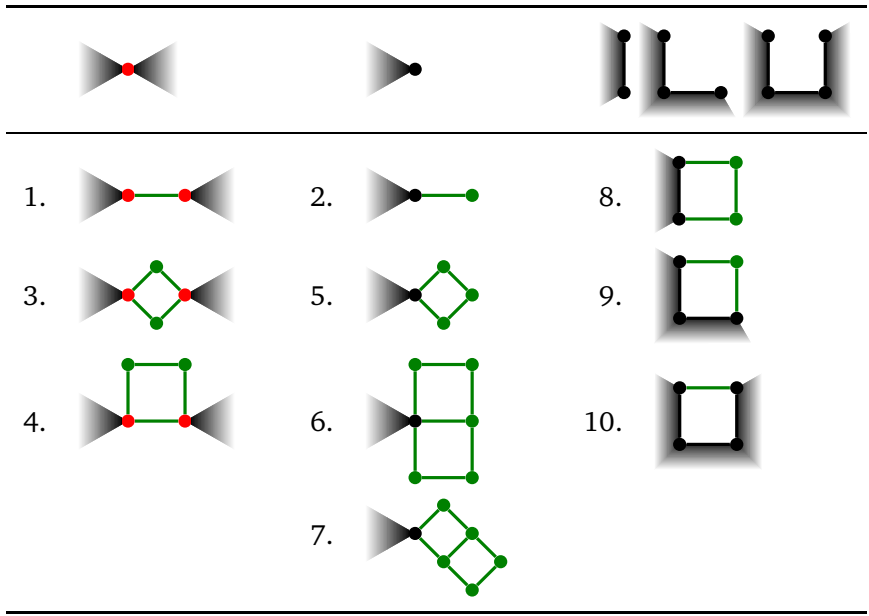

Given a predecoration $P$, we will choose a \emph{canonical parent} of $P$. This is a predecoration obtained by applying one of the reductions to $P$. We will always use the reduction with the smallest number among all possible reductions. It is possible that there is more than one way to apply this reduction to $P$, and if $P$ has non-trivial symmetry, some of these can result in the same parent. If we choose one special edge in the subgraph that is affected by the reduction operation, each way to apply this reduction corresponds to an edge of $P$. We can choose an orbit of edges under the symmetry group of $P$ by constructing a canonical labeling of the vertices -- similar to \cite{plantri} -- and choosing the orbit of the edge with the lowest numbered vertices. The canonical parent of $P$ is then obtained by applying the corresponding reduction.

During the construction, we will try each possible extension in all possible ways, and then check if it is the inverse of the reduction used to get the canonical parent of the resulting predecoration. If that is the case, we can continue to extend this predecoration.

It is possible to construct all predecorations with fewer extensions, but it is important that a canonical reduction always results in a valid predecoration. The order of extensions 1--4 ensures that a canonical reduction never increases $n_A$, and extensions 5--7 ensure that a canonical reduction never increases $n_A + n_B + n_C$. Extensions 8--10 are necessary when none of the other reductions are possible, so that each predecoration different from the base decorations has a possible reduction. We will prove this in \cref{lemma:reduction,thm:predecorations}.

\begin{lemma}\label{lemma:extension}
  An extension applied to a predecoration results in another predecoration if it keeps $n_A \leq 2$ and $n_A + n_B + n_C \leq 3$. Only extensions 1, 2 and 5 possibly violate this condition.
\end{lemma}
\begin{proof}
  It is easy to see that each extension can only create new inner faces that are quadrangles, and inner vertices with degree at least 3.

  The only extensions that can increase $n_A$ are extensions 1 and 2. The only extension that can increase $n_B$ is extension 2. The only extension that can increase $n_C$ is extension 5.
\end{proof}

This makes it easier to keep count of $n_A$, $n_B$ and $n_C$ during the construction.

\begin{lemma}\label{lemma:reduction}
  Let $P$ be a predecoration different from the base predecorations. By applying one of the reductions from \cref{fig:extensions}, $P$ can be reduced to a graph containing fewer vertices or a graph containing the same number of vertices but fewer edges.

  Furthermore, if we apply the reduction with the smallest number among all possible reductions, the resulting graph is again a predecoration.
\end{lemma}
\begin{proof}
  For the first part, it is clear that each reduction results in a `smaller' graph, so we only need to verify that at least one reduction can be applied. If $P$ contains at least one quadrangle, there is at least one quadrangle $Q$ with an edge in the outer face. Since $P$ is not $C_4$, there is at least one other vertex not contained in $Q$ in the graph, and reduction 10 is possible. If there is no quadrangle in $P$, reduction 1 is possible.

  For the second part, it is immediately clear that all reductions preserve the properties that all inner faces are quadrangles and that all inner vertices have degree at least 3. It remains to be proven that for the new graph $n_A \leq 2$ and $n_A + n_B + n_C \leq 3$.

  Some reductions can increase $n_A$, $n_B$ or $n_C$, but only if another reduction with a smaller number can also be applied. This is the reason that we need so many extension operations in that particular order. In \cref{tab:reductions}, all these situations are given.

  \begin{table}\label{tab:reductions}
    \centering
    \begin{tabular}{cccc}
      \toprule
      reduction & $n_A$ & $n_B$ & $n_C$ \\
      \midrule
      1     &     & $n_A$ &       \\
      2     & 1   & 1     & $n_B$ \\
      3,4   & 1   &       &      \\
      5,6,7 & 1   & 1     & 3/4  \\
      8     & 2/5 & 5     & 6/7  \\
      9     & 2/8 & 8     & 8    \\
      10    & 2/9 & 9     & 9    \\
      \bottomrule
    \end{tabular}
    \caption[Table with possible reductions.]{Read this table as: \\[0.5\baselineskip]
    \begin{minipage}{\linewidth}
      \centering\small
      Reduction $i$ can increase $n_X$, but only if $n_Y$ is decreased by the same amount. \\
      Reduction $i$ can increase $n_X$, but only if reduction $j/k$ can be applied too.
    \end{minipage}
    }
  \end{table}

  It is impossible to increase $n_A$ with a reduction that has the smallest possible number. Therefore, we still have $n_A \leq 2$ in the new graph.

  Reduction 1 can increase $n_B$, but only by removing a vertex of degree 2 neighbouring a vertex of degree 1, \ie by decreasing $n_A$ by the same amount. Therefore, we still have $n_A + n_B + n_C \leq 3$ in the new graph.

  Reduction 2 can increase $n_C$, but only by decreasing $n_B$ by the same amount. Therefore, we still have $n_A + n_B + n_C \leq 3$ in the new graph.

\end{proof}

\begin{algorithm}[htp]
	\begin{algorithmic}[0]
    \Function{Extend}{$P$}
      \State \textbf{output} $P$
      \For{$i = 1, \dotsc, 10$}
        \For{$O$ an orbit of edges in the outer face of $P$}
          \State $e \gets \text{edge in } O$
          \State $P' \gets \text{apply extension $i$ to edge $e$ of $P$}$
          \If{$P$ canonical parent of $P'$}
            \State $\textsc{Extend}(P')$
          \EndIf
        \EndFor
      \EndFor
    \EndFunction
    \For{$G$ a base predecoration}
      \State $\textsc{Extend}(P)$
    \EndFor
	\end{algorithmic}
	\caption{Construction of predecorations}\label{alg:predeco}
\end{algorithm}

\begin{theorem}\label{thm:predecorations}
  The algorithm described in \cref{alg:predeco} generates all predecorations.
\end{theorem}
\begin{proof}
  This follows immediately from \cite{McKay1998} and \cref{lemma:reduction}.
\end{proof}

\section{Construction of decorations}

Now that we can construct all predecorations, we can use the homomorphism principle \cite{Gruner1997} and complete each predecoration in all possible ways to get all $k$-decorations with \cref{alg:deco}. We first have to compute the symmetry group of the predecoration, in order to avoid completions that result in the same decoration. After the first 4 steps, all symmetry is broken by choosing $v_0$, $v_1$ and $v_2$.

\begin{algorithm}[htp]
  \begin{enumerate}
    \item If $n_A > 0$, label the corresponding vertices of degree 1 with $v_0$ or $v_2$ in all non-isomorphic ways.
    \item If $n_B + n_C > 0$, label the corresponding vertices with $v_0$, $v_1$ or $v_2$ in all non-isomorphic ways.
    \item If $v_1$ is not yet chosen, label an outer vertex with $v_1$ or add a new type-1 vertex $v_1$ of degree 2 in the outer face in all non-isomorphic ways.
    \item If $v_0$ or $v_2$ is not yet chosen, label two outer vertices with $v_0$ and $v_2$ in all non-isomorphic ways.
    \item Fill all inner quadrangles with a type-1 vertex.
    \item \label{step:6} Add type-1 vertices of degree 3 in the outer face in all possible ways, such that there are no cut-vertices or vertices of degree 2 left.
    \item \label{step:7} Check whether the result is a $k$-decoration.
  \end{enumerate}
	\caption{Complete a predecoration in all possible ways}\label{alg:deco}
\end{algorithm}

We do not have to take isomorphisms into account, since two isomorphic decorations will have isomorphic predecorations.

Note that it might not be possible to complete a predecoration in Step \ref{step:6} such that there are no cut-vertices left.

\subsection{Connectivity}

In Step \ref{step:7}, we will always obtain a decoration. The additional properties for 2\dash connected decorations and 3\dash connected decorations have to be checked. The properties in the outer face cannot be checked earlier in the construction process, because they depend on the chosen completion. But we can prevent type-1 cycles of length 2 and cycles of length 4 during the construction. It is clear that once a type-1 cycle is created during the construction, it cannot be destroyed later. So we only have to avoid the creation of the first type-1 cycle of length 2 or 4.

The only way to create a first type-1 cycle of length 2 is by applying extension 10 to a predecoration with an outer face of size 4. This can easily be avoided. The only way to create a non-empty type-1 cycle of length 4 is by applying extension 10 to a predecoration with an outer face of size 6. We can avoid this too.

To check the other properties after the completion, we can loop over the outer face of the decoration, and mark all vertices one inner edge away from side~$i$ with~$i$. If we encounter a vertex on side $i$ that is marked with $i$, the decoration is not 2\dash connected. If a vertex is marked two times with the same number, or a vertex on side 1 is marked with 0 or vice versa, the decoration is not 3\dash connected.

\subsection{Inflation rate}

As mentioned in \cite{Brinkmann2017}, the impact of an operation on the size of a polyhedron can be measured by the \emph{inflation rate}. This is the ratio of the number of edges before and after the operation, and is equal to the number of chambers in the decoration.

Although it is interesting to construct all possible decorations, we are more interested in the decorations with a given inflation rate. Unfortunately, we cannot determine the inflation rate before the predecoration is completed as decorations with different inflation rates might have the same predecoration, but we can compute lower and upper bounds.

Given a predecoration $P$, for each decoration that has $P$ as its underlying predecoration, each quadrangle of $P$ corresponds to 4 chambers and each cut-vertex of which the removal leaves $k \geq 2$ components requires $2 (k - 1)$ extra chambers. So
\begin{equation}
  4 \cdot (\text{number of quadrangles}) + 2 \cdot \sum_{\text{cut-vertices}} (\text{occurences in outer face} - 1)
\end{equation}
is a lower bound for the inflation rate. The maximal inflation rate of a predecoration is reached by adding as much type-1 vertices as possible in the outer face. This will result in exactly one chamber for each edge in the outer face. In combination with the 4 chambers in each quadrangle, this results in 2 chambers (one at each side) for each edge of the predecoration. So the maximal inflation rate is
\begin{equation}
  2 \cdot (\text{number of edges}).
\end{equation}

If the lower bound for the inflation rate of a predecoration is already higher than the desired inflation rate, we do not have to extend it further as it can only increase. If the upper bound is lower than the desired inflation rate, we have to extend it, but we do not have to try to complete it.

\section{Results}

Using \cref{alg:predeco,alg:deco}, we implemented a computer program \cite{decogen} to generate all $k$-decorations with a given inflation rate. The results of this program are given in \cref{tbl:decogen}. The decorations for inflation rates $r \leq 8$ are given in \cref{tbl:examples}.

The two lsp operations with inflation rate 1 are obviously identity and dual. The lsp operations with inflation rate 2 are ambo and join, and the ones with inflation rate 3 are truncate, zip, needle and kiss. Up to here, all lsp operations were already described by Conway \cite{conway} or others. For the left decoration with inflation rate 4, only two of the 4 related lsp operations (chamfer and subdivide) are already named. The first decoration for which none of the related lsp operations (including dual and mirrored ones) are already named, is the 2\dash connected lsp operation with inflation rate 5. The first unnamed 3\dash connected lsp operations are the three leftmost decorations with inflation rate 6.

These results are verified for inflation rate up to 23 by an independent implementation that constructs all triangulations, filters the decorations out, applies them to a polyhedron, checks the connectivity and filters the isomorphic ones out.

\begin{table}
  \small
  \centering
  \begin{tabular}{crrrr}
    \toprule
    & \multicolumn{3}{c}{$k$\dash connected decorations} & \\
    \cmidrule{2-4}
    inflation rate & \multicolumn{1}{c}{$k = 1$} & \multicolumn{1}{c}{$k = 2$} & \multicolumn{1}{c}{$k = 3$} & \multicolumn{1}{c}{predecorations} \\
    \midrule
    1 & 2 & 2 & 2 & 1 \\
    2 & 2 & 2 & 2 & 1 \\
    3 & 4 & 4 & 4 & 1 \\
    4 & 6 & 6 & 6 & 2 \\
    5 & 6 & 6 & 4 & 2 \\
    6 & 20 & 20 & 20 & 4 \\
    7 & 28 & 28 & 20 & 7 \\
    8 & 58 & 58 & 54 & 8 \\
    9 & 82 & 82 & 64 & 7 \\
    10 & 170 & 168 & 144 & 19 \\
    11 & 204 & 200 & 132 & 16 \\
    12 & 496 & 492 & 404 & 50 \\
    13 & 650 & 640 & 396 & 42 \\
    14 & 1432 & 1400 & 1112 & 118 \\
    15 & 1824 & 1786 & 1100 & 109 \\
    16 & 4114 & 3952 & 2958 & 298 \\
    17 & 5078 & 4900 & 2769 & 300 \\
    18 & 11874 & 11150 & 7972 & 749 \\
    19 & 14808 & 14058 & 7560 & 782 \\
    20 & 33978 & 30998 & 21300 & 1902 \\
    21 & 41794 & 38964 & 20076 & 2056 \\
    22 & 97096 & 85976 & 56296 & 4893 \\
    23 & 118572 & 107784 & 52380 & 5419 \\
    24 & 277208 & 237482 & 148956 & 12615 \\
    25 & 337216 & 298546 & 138384 & 14153 \\
    26 & 788342 & 652236 & 392096 & 32665 \\
    27 & 953060 & 820960 & 362499 & 36953 \\
    28 & 2239396 & 1786222 & 1027488 & 84853 \\
    29 & 2697088 & 2250816 & 945612 & 96491 \\
    30 & 6350014 & 4875076 & 2687408 & 220646 \\
    31 & 7618068 & 6153604 & 2466156 & 251104 \\
    32 & 17972390 & 13262574 & 7007118 & 573547 \\
    33 & 21487746 & 16773086 & 6409664 & 654663 \\
    34 & 50805716 & 35985748 & 18222032 & 1491540 \\
    35 & 60573248 & 45592594 & 16623268 & 1706755 \\
    36 & 143425040 & 97394726 & 47287986 & 3878836 \\
    37 & 170530518 & 123628298 & 43038260 & 4446426 \\
    38 & 404413576 & 262983002 & 122451618 & 10085305 \\
    39 & 479711448 & 334473144 & 111200316 & 11582891 \\
    40 & 1139138344 & 708583784 & 316474370 & 26222191 \\
    \bottomrule
  \end{tabular}
  \caption{The number of $k$-connected decorations up to inflation rate 40. The number of predecorations that can be completed to a decoration with given inflation rate are given too. Not all of these predecorations are constructed for 2-connected or 3-connected decorations.}\label{tbl:decogen}
\end{table}

\begin{table}\label{tbl:examples}
  \newcommand{\Lcomplete}[2]{\draw[edge,thick] (#1) -- (0,0) -- (#2);}
  \newcommand{\Tcomplete}[3]{\draw[edge,thick] (#1) -- (#3) (#2) -- ($(#1)!(#2)!(#3)$);}
  \newcommand{\Xcomplete}[4]{\draw[edge,thick] (#1) -- (#3) (#2) -- (#4);}
  \newcommand{\symmetric}{\node[fill,star,inner sep=0,minimum size=6,star point ratio=2.5] at (2.75,2.75) {};}
  \begin{tabular}{cll}
    \toprule
    inflation rate & $k = 2$ & $k = 3$ \\
    \midrule \addlinespace
    1 & &
    \begin{tikzpicture}[scale=0.3,baseline=2.5ex]
      \symmetric
      \coordinate (1) at (4,0);
      \coordinate (2) at (0,4);
      \Lcomplete 1 2
      \draw[edge1,very thick] (1) -- (2);
    \end{tikzpicture} \\ \addlinespace
    2 & &
    \begin{tikzpicture}[scale=0.3,baseline=2.5ex]
      \symmetric
      \coordinate (1) at (4,0);
      \coordinate (2) at (0,0);
      \coordinate (3) at (0,4);
      \Tcomplete 1 2 3
      \draw[edge1,very thick] (1) -- (2) -- (3);
    \end{tikzpicture} \\ \addlinespace
    3 & &
    \begin{tikzpicture}[scale=0.3,baseline=2.5ex]
      \coordinate (1) at (4,0);
      \coordinate (2) at (2,0);
      \coordinate (3) at (0,4);
      \Lcomplete 2 3
      \Tcomplete 1 2 3
      \draw[edge1,very thick] (1) -- (2) -- (3);
    \end{tikzpicture} \\ \addlinespace
    4 & &
    \begin{tikzpicture}[scale=0.3,baseline=2.5ex]
      \coordinate (1) at (4,0);
      \coordinate (2) at ($(4,0)!0.3!(0,4)$);
      \coordinate (3) at (0,0);
      \coordinate (4) at (0,4);
      \Tcomplete 1 2 3
      \Tcomplete 2 3 4
      \draw[edge1,very thick] (1) -- (2) -- (3) -- (4);
    \end{tikzpicture}
    \begin{tikzpicture}[scale=0.3,baseline=2.5ex]
      \symmetric
      \coordinate (1) at (4,0);
      \coordinate (2) at (2,2);
      \coordinate (3) at (0,0);
      \coordinate (4) at (0,4);
      \Tcomplete 1 2 3
      \Tcomplete 3 2 4
      \draw[edge1,very thick] (1) -- (4) (2) -- (3);
    \end{tikzpicture} \\ \addlinespace
    5 &
    \begin{tikzpicture}[scale=0.3,baseline=2.5ex]
      \symmetric
      \coordinate (1) at (1.5,0);
      \coordinate (2) at (0,1.5);
      \coordinate (3) at (4,0);
      \coordinate (4) at (0,4);
      \Lcomplete 1 2
      \Xcomplete 1 2 4 3
      \draw[edge1,very thick] (1) -- (2) -- (4) -- (3) -- (1);
    \end{tikzpicture} &
    \begin{tikzpicture}[scale=0.3,baseline=2.5ex]
      \coordinate (1) at (4,0);
      \coordinate (2) at (3,1);
      \coordinate (3) at (1,0);
      \coordinate (4) at (0,4);
      \Lcomplete 3 4
      \Tcomplete 1 2 3
      \Tcomplete 2 3 4
      \draw[edge1,very thick] (1) -- (2) -- (3) -- (4);
    \end{tikzpicture} \\ \addlinespace
    6 & &
    \begin{tikzpicture}[scale=0.3,baseline=2.5ex]
      \coordinate (1) at (4,0);
      \coordinate (2) at (3,0);
      \coordinate (3) at ($(4,0)!0.4!(0,4)$);
      \coordinate (4) at (0,0);
      \coordinate (5) at (0,4);
      \Tcomplete 1 2 3
      \Tcomplete 2 3 4
      \Tcomplete 3 4 5
      \draw[edge1,very thick] (1) -- (2) -- (3) -- (4) -- (5);
    \end{tikzpicture}
    \begin{tikzpicture}[scale=0.3,baseline=2.5ex]
      \symmetric
      \coordinate (1) at (4,0);
      \coordinate (2) at ($(4,0)!0.3!(0,4)$);
      \coordinate (3) at (0,0);
      \coordinate (4) at ($(4,0)!0.7!(0,4)$);
      \coordinate (5) at (0,4);
      \Tcomplete 1 2 3
      \Tcomplete 2 3 4
      \Tcomplete 3 4 5
      \draw[edge1,very thick] (1) -- (2) -- (3) -- (4) -- (5);
    \end{tikzpicture}
    \begin{tikzpicture}[scale=0.3,baseline=2.5ex]
      \coordinate (1) at ($(4,0)!0.33!(0,4)$);
      \coordinate (2) at (2,0);
      \coordinate (3) at (0,4);
      \coordinate (4) at (4,0);
      \coordinate (5) at (0,0);
      \Tcomplete 2 1 4
      \Xcomplete 1 2 5 3
      \draw[edge1,very thick] (1) -- (2) -- (5) -- (3) -- (1) -- (4);
    \end{tikzpicture}
    \begin{tikzpicture}[scale=0.3,baseline=2.5ex]
      \symmetric
      \coordinate (1) at (4,0);
      \coordinate (2) at (1.33,1.33);
      \coordinate (3) at (0,0);
      \coordinate (4) at (0,4);
      \Tcomplete 1 2 3
      \Tcomplete 1 2 4
      \Tcomplete 3 2 4
      \draw[edge1,very thick] (1) -- (2) -- (3) (2) -- (4);
    \end{tikzpicture}
    \begin{tikzpicture}[scale=0.3,baseline=2.5ex]
      \coordinate (1) at (2.5,0);
      \coordinate (2) at (0,0);
      \coordinate (3) at (2.5,1.5);
      \coordinate (4) at (4,0);
      \coordinate (5) at (0,4);
      \Tcomplete 3 1 4
      \Xcomplete 1 2 5 3
      \draw[edge1,very thick] (1) -- (2) -- (5) -- (3) -- (1) -- (4);
    \end{tikzpicture}
    \begin{tikzpicture}[scale=0.3,baseline=2.5ex]
      \coordinate (1) at (0,0);
      \coordinate (2) at (3,1);
      \coordinate (3) at (0,2);
      \coordinate (4) at (4,0);
      \coordinate (5) at (0,4);
      \Tcomplete 1 2 3
      \Tcomplete 1 2 4
      \Tcomplete 2 3 5
      \draw[edge1,very thick] (1) -- (2) -- (4) (2) -- (3) -- (5);
    \end{tikzpicture} \\ \addlinespace
    7 &
    \begin{tikzpicture}[scale=0.3,baseline=2.5ex]
      \coordinate (1) at (3,0);
      \coordinate (2) at (1.5,0);
      \coordinate (3) at (0,4);
      \coordinate (4) at (4,0);
      \coordinate (5) at (0,1.5);
      \Tcomplete 3 1 4
      \Xcomplete 1 2 5 3
      \Lcomplete 2 5
      \draw[edge1,very thick] (1) -- (2) -- (5) -- (3) -- (1) -- (4);
    \end{tikzpicture}
    \begin{tikzpicture}[scale=0.3,baseline=2.5ex]
      \coordinate (1) at (0,4);
      \coordinate (2) at (2.5,0);
      \coordinate (3) at (0,1.25);
      \coordinate (4) at (4,0);
      \coordinate (5) at (1.25,0);
      \draw[edge1,very thick] (1) -- (2) -- (5) -- (3) -- (1) -- (4);
      \draw[edge] (2) -- (4) (1) -- (3.25,0);
      \Xcomplete 1 2 5 3
      \Lcomplete 3 5
    \end{tikzpicture} &
    \begin{tikzpicture}[scale=0.3,baseline=2.5ex]
      \coordinate (1) at (4,0);
      \coordinate (2) at (3.25,0);
      \coordinate (3) at ($(4,0)!0.3!(0,4)$);
      \coordinate (4) at (0.75,0);
      \coordinate (5) at (0,4);
      \Tcomplete 1 2 3
      \Tcomplete 2 3 4
      \Tcomplete 3 4 5
      \Lcomplete 4 5
      \draw[edge1,very thick] (1) -- (2) -- (3) -- (4) -- (5);
    \end{tikzpicture}
    \begin{tikzpicture}[scale=0.3,baseline=2.5ex]
      \coordinate (1) at (2,0);
      \coordinate (2) at (0,1);
      \coordinate (3) at (2,2);
      \coordinate (4) at (4,0);
      \coordinate (5) at (0,4);
      \Tcomplete 3 1 4
      \Xcomplete 1 2 5 3
      \Lcomplete 1 2
      \draw[edge1,very thick] (1) -- (2) -- (5) -- (3) -- (1) -- (4);
    \end{tikzpicture}
    \begin{tikzpicture}[scale=0.3,baseline=2.5ex]
      \coordinate (1) at (2.5,0);
      \coordinate (2) at (0.75,0);
      \coordinate (3) at ($(4,0)!0.33!(0,4)$);
      \coordinate (4) at (4,0);
      \coordinate (5) at (0,4);
      \Tcomplete 3 1 4
      \Xcomplete 1 2 5 3
      \Lcomplete 2 5
      \draw[edge1,very thick] (1) -- (2) -- (5) -- (3) -- (1) -- (4);
    \end{tikzpicture}
    \begin{tikzpicture}[scale=0.3,baseline=2.5ex]
      \coordinate (1) at (2,2);
      \coordinate (2) at (1,0);
      \coordinate (3) at (0,4);
      \coordinate (4) at (4,0);
      \coordinate (5) at (0,1.5);
      \Tcomplete 2 1 4
      \Xcomplete 1 2 5 3
      \Lcomplete 2 5
      \draw[edge1,very thick] (1) -- (2) -- (5) -- (3) -- (1) -- (4);
    \end{tikzpicture}
    \begin{tikzpicture}[scale=0.3,baseline=2.5ex]
      \coordinate (1) at ($(4,0)!0.3!(0,4)$);
      \coordinate (2) at (2,0);
      \coordinate (3) at (0,4);
      \coordinate (4) at (4,0);
      \coordinate (5) at (0.75,0);
      \Tcomplete 2 1 4
      \Xcomplete 1 2 5 3
      \Lcomplete 3 5
      \draw[edge1,very thick] (1) -- (2) -- (5) -- (3) -- (1) -- (4);
    \end{tikzpicture} \\ \addlinespace
    8 &
    \begin{tikzpicture}[scale=0.3,baseline=2.5ex]
      \coordinate (1) at (1,1);
      \coordinate (2) at (0,4);
      \coordinate (3) at (2,0);
      \coordinate (4) at (0,0);
      \coordinate (5) at (4,0);
      \Tcomplete 2 1 4
      \Tcomplete 3 1 4
      \Xcomplete 1 2 5 3
      \draw[edge1,very thick] (1) -- (2) -- (5) -- (3) -- (1) -- (4);
    \end{tikzpicture} &
    \begin{tikzpicture}[scale=0.3,baseline=2.5ex]
      \coordinate (1) at (4,0);
      \coordinate (2) at ($(4,0)!0.1!(0,4)$);
      \coordinate (3) at (2.75,0);
      \coordinate (4) at ($(4,0)!0.42!(0,4)$);
      \coordinate (5) at (0,0);
      \coordinate (6) at (0,4);
      \Tcomplete 1 2 3
      \Tcomplete 2 3 4
      \Tcomplete 3 4 5
      \Tcomplete 4 5 6
      \draw[edge1,very thick] (1) -- (2) -- (3) -- (4) -- (5) -- (6);
    \end{tikzpicture}
    \begin{tikzpicture}[scale=0.3,baseline=2.5ex]
      \coordinate (1) at (4,0);
      \coordinate (2) at (3,0);
      \coordinate (3) at ($(4,0)!0.4!(0,4)$);
      \coordinate (4) at (0,0);
      \coordinate (5) at ($(4,0)!0.75!(0,4)$);
      \coordinate (6) at (0,4);
      \Tcomplete 1 2 3
      \Tcomplete 2 3 4
      \Tcomplete 3 4 5
      \Tcomplete 4 5 6
      \draw[edge1,very thick] (1) -- (2) -- (3) -- (4) -- (5) -- (6);
    \end{tikzpicture}
    \begin{tikzpicture}[scale=0.3,baseline=2.5ex]
      \coordinate (1) at (4,0);
      \coordinate (2) at ($(4,0)!0.2!(0,4)$);
      \coordinate (3) at (0,0);
      \coordinate (4) at (0,1.5);
      \coordinate (5) at ($(4,0)!0.8!(0,4)$);
      \coordinate (6) at (0,4);
      \Tcomplete 1 2 3
      \Tcomplete 3 2 4
      \Tcomplete 2 4 5
      \Tcomplete 4 5 6
      \draw[edge1,very thick] (1) -- (2) -- (3) (2) -- (4) -- (5) -- (6);
    \end{tikzpicture}
    \begin{tikzpicture}[scale=0.3,baseline=2.5ex]
      \symmetric
      \coordinate (1) at (2,2);
      \coordinate (2) at (0,3);
      \coordinate (3) at (0,4);
      \coordinate (4) at (3,0);
      \coordinate (5) at (4,0);
      \coordinate (6) at (0,0);
      \Tcomplete 1 2 3
      \Tcomplete 1 4 5
      \Tcomplete 2 1 6
      \Tcomplete 4 1 6
      \draw[edge1,very thick] (3) -- (2) -- (1) -- (4) -- (5) (1) -- (6);
    \end{tikzpicture}
    \begin{tikzpicture}[scale=0.3,baseline=2.5ex]
      \coordinate (1) at (2,0);
      \coordinate (2) at (0,0);
      \coordinate (3) at ($(4,0)!0.35!(0,4)$);
      \coordinate (4) at ($(4,0)!0.15!(0,4)$);
      \coordinate (5) at (0,4);
      \coordinate (6) at (4,0);
      \Tcomplete 3 1 4
      \Tcomplete 1 4 6
      \Xcomplete 1 2 5 3
      \draw[edge1,very thick] (1) -- (2) -- (5) -- (3) -- (1) -- (4) -- (6);
    \end{tikzpicture}
    \begin{tikzpicture}[scale=0.3,baseline=2.5ex]
      \coordinate (1) at ($(4,0)!0.35!(0,4)$);
      \coordinate (2) at (2,0);
      \coordinate (3) at (0,4);
      \coordinate (4) at (3.1,0);
      \coordinate (5) at (0,0);
      \coordinate (6) at (4,0);
      \Tcomplete 2 1 4
      \Tcomplete 1 4 6
      \Xcomplete 1 2 5 3
      \draw[edge1,very thick] (1) -- (2) -- (5) -- (3) -- (1) -- (4) -- (6);
    \end{tikzpicture} \\ \addlinespace & &
    \begin{tikzpicture}[scale=0.3,baseline=2.5ex]
      \coordinate (1) at (0,1.25);
      \coordinate (2) at (0,0);
      \coordinate (3) at (1.25,0);
      \coordinate (4) at (2,2);
      \coordinate (5) at (0,4);
      \coordinate (6) at (4,0);
      \Tcomplete 4 1 5
      \Tcomplete 3 4 6
      \Xcomplete 1 2 3 4
      \draw[edge1,very thick] (1) -- (2) -- (3) -- (4) -- (1) -- (5) (4) -- (6);
    \end{tikzpicture}
    \begin{tikzpicture}[scale=0.3,baseline=2.5ex]
      \coordinate (1) at (3,1);
      \coordinate (2) at (2,0);
      \coordinate (3) at (0,0);
      \coordinate (4) at (1,3);
      \coordinate (5) at (4,0);
      \coordinate (6) at (0,4);
      \Tcomplete 2 1 5
      \Tcomplete 3 4 6
      \Xcomplete 1 2 3 4
      \draw[edge1,very thick] (1) -- (2) -- (3) -- (4) -- (1) -- (5) (4) -- (6);
    \end{tikzpicture}
    \begin{tikzpicture}[scale=0.3,baseline=2.5ex]
      \coordinate (1) at (2.8,0);
      \coordinate (2) at (0,0);
      \coordinate (3) at ($(4,0)!0.41!(0,4)$);
      \coordinate (4) at ($(4,0)!0.23!(0,4)$);
      \coordinate (5) at (4,0);
      \coordinate (6) at (0,4);
      \Tcomplete 4 1 5
      \Tcomplete 3 2 6
      \Xcomplete 1 2 3 4
      \draw[edge1,very thick] (1) -- (2) -- (3) -- (4) -- (1) -- (5) (2) -- (6);
    \end{tikzpicture}
    \begin{tikzpicture}[scale=0.3,baseline=2.5ex]
      \coordinate (1) at (0,1.8);
      \coordinate (2) at (0,0);
      \coordinate (3) at (1.25,0);
      \coordinate (4) at ($(4,0)!0.4!(0,4)$);
      \coordinate (5) at (0,4);
      \coordinate (6) at (4,0);
      \Tcomplete 4 1 5
      \Tcomplete 3 4 6
      \Xcomplete 1 2 3 4
      \draw[edge1,very thick] (1) -- (2) -- (3) -- (4) -- (1) -- (5) (4) -- (6);
    \end{tikzpicture}
    \begin{tikzpicture}[scale=0.3,baseline=2.5ex]
      \coordinate (1) at (2.65,0);
      \coordinate (2) at (0,0);
      \coordinate (3) at ($(4,0)!0.41!(0,4)$);
      \coordinate (4) at ($(4,0)!0.19!(0,4)$);
      \coordinate (5) at (4,0);
      \coordinate (6) at (0,4);
      \Tcomplete 1 4 5
      \Tcomplete 3 2 6
      \Xcomplete 1 2 3 4
      \draw[edge1,very thick] (1) -- (2) -- (3) -- (4) -- (1) (4) -- (5) (2) -- (6);
    \end{tikzpicture}
    \begin{tikzpicture}[scale=0.3,baseline=2.5ex]
      \symmetric
      \coordinate (1) at (2,0);
      \coordinate (2) at (0,0);
      \coordinate (3) at (0,2);
      \coordinate (4) at (2,2);
      \coordinate (5) at (4,0);
      \coordinate (6) at (0,4);
      \Tcomplete 4 1 5
      \Tcomplete 4 3 6
      \Xcomplete 1 2 3 4
      \draw[edge1,very thick] (1) -- (2) -- (3) -- (4) -- (1) -- (5) (3) -- (6);
    \end{tikzpicture} \\ \addlinespace & &
    \begin{tikzpicture}[scale=0.3,baseline=2.5ex]
      \symmetric
      \coordinate (1) at (2,2);
      \coordinate (2) at (1.25,0);
      \coordinate (3) at (0,0);
      \coordinate (4) at (0,1.25);
      \coordinate (5) at (4,0);
      \coordinate (6) at (0,4);
      \Tcomplete 2 1 5
      \Tcomplete 4 1 6
      \Xcomplete 1 2 3 4
      \draw[edge1,very thick] (6) -- (1) -- (2) -- (3) -- (4) -- (1) -- (5);
    \end{tikzpicture}
    \begin{tikzpicture}[scale=0.3,baseline=2.5ex]
      \coordinate (1) at (2,0.66);
      \coordinate (2) at (1.33,0);
      \coordinate (3) at (0,0);
      \coordinate (4) at (0,4);
      \coordinate (5) at (4,0);
      \Tcomplete 2 1 5
      \Tcomplete 4 1 5
      \Xcomplete 1 2 3 4
      \draw[edge1,very thick] (1) -- (2) -- (3) -- (4) -- (1) -- (5);
    \end{tikzpicture}
    \begin{tikzpicture}[scale=0.3,baseline=2.5ex]
      \coordinate (1) at (2.5,0.5);
      \coordinate (2) at (0,0);
      \coordinate (3) at (0,4);
      \coordinate (4) at ($(4,0)!0.4!(0,4)$);
      \coordinate (5) at (4,0);
      \Tcomplete 2 1 5
      \Tcomplete 4 1 5
      \Xcomplete 1 2 3 4
      \draw[edge1,very thick] (1) -- (2) -- (3) -- (4) -- (1) -- (5);
    \end{tikzpicture} \\ \addlinespace
    \bottomrule
  \end{tabular}
  \caption{All decorations with inflation rate up to 8. The \iftoggle{bw}{dashed}{green} lines are edges of type 1. The \iftoggle{bw}{solid}{black} lines are edges of type 0 and 2. For each of the given decorations, the edges of type 0 and 2 can be chosen in two different ways. All decorations except the symmetric ones (marked with a star) can be mirrored. So each starred decoration represents two related lsp operations, and the unstarred ones represent four related lsp operations.}
\end{table}

\clearpage

\nocite{*}
\printbibliography

\end{document}